\documentclass{article}
\usepackage{amsfonts,enumerate}
\usepackage{amsthm,fullpage}
\usepackage{amsmath,mathdots}
\usepackage{amssymb,mathrsfs}
\newcommand{\al}{\alpha}
\newcommand{\be}{\beta}
\newcommand{\si}{\sigma}
\newcommand{\Ckh}{\mathcal{C}_{kh}}
\newcommand{\Wkh}{\mathcal{W}_{kh}}
\newcommand{\Rkh}{\mathcal{R}_{kh}}
\newcommand{\la}{\lambda}
\newcommand{\R}{\mathbb{R}}
\newcommand{\p}{\partial}

\newcommand{\F}{\mathcal{F}}
\renewcommand{\cal}{\mathcal}
\DeclareMathOperator{\im}{\mathrm{im}}
\DeclareMathOperator{\re}{\mathrm{Re}}
\DeclareMathOperator{\imag}{\mathrm{Im}}

\DeclareMathOperator{\sgn}{\mathrm{sgn}}
\newtheoremstyle{NTS1}{\topsep}{\topsep}%
     {\it}
     {}
     {\bf}
     {}
     { }
     {\thmname{#1}\thmnumber{ #2}\thmnote{ (#3)}}

\theoremstyle{NTS1}
\newtheorem{thm}{Theorem}[section]

\newtheorem{lem}[thm]{Lemma}

\theoremstyle{remark}
\newtheorem*{prems}{Remarks}
\newenvironment{rems}{\begin{prems}$\;$\begin{enumerate}\item}{\end{enumerate}\end{prems}}
\theoremstyle{remark}
\newtheorem*{rem}{Remark}
\numberwithin{equation}{section}
\begin{document}

\title{Existence of capillary-gravity waves that are perturbations of Crapper's waves}
\author{Peter de Boeck}
\date{}
\maketitle
\begin{abstract}
This paper is concerned with two-dimensional, steady, periodic water waves propagating at the free surface of water either in a flow of finite depth and constant vorticity over an impermeable flat bed or in an irrotational flow of great depth. In both cases, the motion of these waves is assumed to be governed both by surface tension and gravitational forces. By considering a particular scaling regime, it can be shown that both these schemes approach a limiting form, corresponding to flows of great depth where the effect of gravity is neglected. An application of the Implicit Function Theorem demonstrates that Crapper's explicit solutions to this latter problem can be perturbed by means of the aforementioned scaling to yield capillary-gravity waves of any finite (or infinite) depth and any constant vorticity.
\end{abstract}
\section{Introduction}
The problem of finding steady periodic pure-capillary waves, travelling on the surface of fluid of great depth can be trivially obtained by formally setting $g=0$ in the equations for the corresponding capillary-gravity problem. That is to say, this problem may be approached by imagining that the force of gravity is non-existent and that surface tension is the only force governing the behaviour of these waves.

Historically, and in more recent years, there have been many studies of this pure-capillary problem, for both the finite- and infinite-depth scenarios (e.g. \cite{Crapper,Kin,Mcap,Mreg,Wcap}). This particular problem is mainly motivated by the observation that in many situations (such as when the wind begins blowing over the flat, still surface of a body of water \cite{Kin}) small-amplitude wave trains can arise that are governed principally by capillarity. As these waves grow in amplitude, the effects of gravity begin to be felt, and the waves are best described by the capillary-gravity problem (cf. the discussion in \cite{Walsh}). It is therefore natural to study how pure-capillary waves may evolve into capillary-gravity waves and in what sense the latter can be viewed as perturbations of the former.

Amongst the most well-known works on pure-capillary waves is the paper \cite{Crapper}, written in 1957 by G.D. Crapper. In this paper, Crapper wrote down a one-parameter family of explicit solutions for the infinite-depth pure-capillary travelling wave problem, which were to become known as Crapper's waves.  These are solutions whose profiles can be written down explicitly and which can have arbitrary amplitude.  These waves exhibit many interesting properties: in addition to having large amplitude, there exist parameter values where the profiles are overhanging and those where `bubbles' of air are trapped inside the fluid.  For large enough values of the parameter, the profiles become self-intersecting, and these profiles are discarded as unrealistic (cf. \cite{OS}).

In view of the earlier comments, when studying capillary-gravity waves as perturbations of pure-capillary waves, one may consider how -- and indeed if -- Crapper's waves are thus perturbed. Mathematically speaking, it can be contemplated whether there exist capillary-gravity waves for any small, non-zero $g\in\R$ that approach Crapper's waves (in some sense) as $g\to0$. Indeed, this is exactly the approach adopted in the recent paper \cite{AAW}, where the answer was given in the affirmative. Via an application of the Implicit Function Theorem, the authors were able to demonstrate the existence of infinite-depth capillary-gravity waves near to Crapper's waves.

In this work, we aim to establish somewhat similar results to those in \cite{AAW}, but using a different approach, which will be seen later to generalise to a broader setting. In particular, we wish to investigate in what manner infinite-depth pure-capillary waves can be seen as a limiting regime, both for infinite-depth capillary-gravity waves and finite-depth capillary-gravity waves with constant non-zero vorticity. Then, we aim to establish that Crapper's waves are, in some as-yet undefined sense, a limiting profile under this same regime.

As will be seen in \S2 (cf. \cite{ST}), the problem of infinite-depth capillary-gravity waves is equivalent to finding solutions of\begin{equation}\left(w'^2+(1+\cal{C}w')^2\right)^{-1/2}\\=\left(\frac{Q}{c^2}-\frac{2gw}{c^2k}\right)\left(w'^2+(1+\cal{C}w')^2\right)^{1/2}+2\frac{\si k}{c^2}\frac{(1+\mathcal{C}w')w''-w'\mathcal{C}w''}{w'^2+(1+\mathcal{C}w')^2},\label{deepcg}\end{equation}where $w$ is an appropriately regular $2\pi$-periodic function, $\cal{C}$ denotes the periodic Hilbert transform and $c,g,k,Q,\si$ denote parameters relating to wave speed, gravity, wave period, hydraulic head and surface tension, respectively. As mentioned previously, we can obtain the pure-capillary problem by setting $g=0$, leading to the equation \begin{equation}\left(w'^2+(1+\cal{C}w')^2\right)^{-1/2}\\=\frac{Q}{c^2}\left(w'^2+(1+\cal{C}w')^2\right)^{1/2}+2\frac{\si k}{c^2}\frac{(1+\mathcal{C}w')w''-w'\mathcal{C}w''}{w'^2+(1+\mathcal{C}w')^2}.\label{deepc}\end{equation}Then, for some $c,k,Q,\si$ such that $Q/c^2=1$ and $\si k/c^2=(1+A^2)/(1-A^2)$ for some parameter $A\in(-1,1)$, we shall show in \S\ref{crapsol} that \eqref{deepc} is satisfied by Crapper's waves, which correspond to the function $$w_A(t)=\frac{2(1-A^2)}{1+A^2+2A\cos{t}}-2;$$the constant arising by virtue of our requirement that $w_A$ have zero mean over one period.

However, it can be noted that gravitational effects are not non-existent, and so we should aim not to consider $g=0$.  With this in mind, our aim is to understand how pure capillary waves can arise as limiting profiles in certain scaling regimes.  For example, when the profile is close to being horizontal (i.e. $w\ll1$), but remains strongly curved (e.g. when the profile is highly oscillatory), it is clear that the surface tension term in \eqref{deepcg} will dominate over the gravitational term; similarly (for any amplitude of profile) when the period becomes sufficiently small (i.e. $k\gg1$). As remarked above, wind-driven waves on the surface of a lake or pond often demonstrate this phenomenon. Thus, it follows that the pure capillary regime is a good \textit{approximation} to the gravity-capillary regime, when considering waves of very short period propagating with very fast speed. (We require the speed of wave propagation to be large under a suitable scaling in order to balance the shortness of the period.)

As was remarked earlier, we shall also be interested in finite-depth capillary-gravity waves with constant vorticity, which we shall see in \S2 (cf. \cite{M}) are governed predominantly by the equation \begin{multline}\left\{\frac{m}{h}-\frac{\gamma h}{2}+\frac\gamma k\left(\frac{[w^2]}{2kh}+\Ckh(ww')-w-w\Ckh w'\right)\right\}^2\\=\left(Q-\frac{2gw}k+2\sigma k\frac{(1+\Ckh w')w''-w'\Ckh w''}{(w'^2+(1+\Ckh w')^2)^{3/2}}\right)(w'^2+(1+\Ckh w')^2),\label{finitecg}\end{multline}where $\cal{C}_d$ is the periodic Hilbert transform for a strip and $\gamma$ is the constant of vorticity. In a sense to be made precise in \S\ref{crapbif}, this operator has the property that $\cal{C}_d\to\cal{C}$ as $d\to\infty$. It is this fact that will allow us to conclude that, under the same scaling (where $c:=m/h$), the equation \eqref{finitecg} takes \eqref{deepc} as its limiting form.

Although we do not wish to make an approximation to our full regime, we can make use of the limiting solutions to prove existence of solutions as we approach the limiting regime, in the following manner. For $g,\si>0$ we consider the following, bijective, change of variables from $\R_+\times\R_+$ to $\R_+\times\R_+$:\footnote{Note that we take $c>0$ here.  When $\gamma=0$ or when we are in the infinite-depth scenario it is clear that the equation depends only on $c^2$ and that this is then no great restriction.}$$(\al,\be)=\left(\frac{g}{c^2k},\frac{\si k}{c^2}\right).$$Then, the infinite-depth gravity-capillary problem arises exactly when $\al=0$, if this change of variables is made in \eqref{deepcg} and \eqref{deepc}. Making the same change of variables in \eqref{finitecg}, it can be shown that \eqref{deepc} arises as the formal limit, when taking $\al\to0$.

By exploiting this fact that \eqref{deepc} occurs as the limiting form for both \eqref{deepcg} and \eqref{finitecg}, we can rigorously prove the existence of capillary-gravity waves of both finite- and infinite-depth, the former also in the case of constant non-zero vorticity; these waves will be seen to be perturbations of Crapper's waves in a sense to be made definite.
\section{The free boundary problem}
Given any $L>0$, a domain $\Omega\subset\R^2$ is called an $L$-periodic strip-like domain if it is horizontally unbounded, with a boundary which consisting of the real axis, denoted by $\cal{B}$, and a (possibly unknown) curve $\cal{S}$, given in parametric form by \begin{subequations}\label{fbdef}\begin{equation}\cal{S}=\{(u(s),v(s)):s\in\R\}\label{surface}\end{equation}such that \begin{equation}u(s+L)=u(s)+L,\quad v(s+L)=v(s)\text{ for all }s\in\R.\label{fbper}\end{equation}\end{subequations}Then, for any $L>0$, the problem of finding $L$-periodic steady gravity-capillary water waves with a constant vorticity $\gamma$ in a flow of finite depth over a flat bed can be formulated as the free-boundary problem of finding an $L$-periodic strip-like domain $\Omega\subset\R^2$ (where the curve $\cal{S}$ represents the free surface of the water and is \emph{a priori} unknown) and a function $\psi:\Omega\to\R$, satisfying the following equations and boundary conditions:\begin{subequations}\label{form}\begin{align}\Delta\psi&=-\gamma\quad\,\,\text{in }\Omega,&\\\psi&=-m\quad\text{on }\cal{B},\label{bedcond}\\\psi&=0\,\qquad\text{on }\cal{S},\\|\nabla\psi|^2+2gv-2\si\frac{u_sv_{ss}-u_{ss}v_s}{(u_s^2+v_s^2)^{3/2}}&=Q\quad\;\;\;\text{on }\cal{S}.\end{align}\end{subequations}The function $\psi$ represents the stream function, giving the velocity field $(\psi_Y,-\psi_X)$ in a frame moving at the constant wave speed; here the parameterisation in \eqref{surface} has been chosen so as to be time-independent in this moving frame. The constants $g,m\in\R,\si>0$ in \eqref{form} are, respectively, the constant of gravitational acceleration, the relative mass flux and the coefficient of surface tension. The constant $Q\in\R$ is a constant related to the hydraulic head.

It is possible to derive these equations directly from the Euler equations for incompressible fluid flow, as in \cite{M}.

We also wish to study in this paper capillary-gravity waves of infinite depth. It was shown in \cite{E} that for flows of infinite depth, the only constant vorticity is zero vorticity, and so we restrict ourselves to irrotational flows, and take $\gamma=0$. In this regime, we search for a domain $\Omega$ lying below a curve $\cal{S}$ satisfying \eqref{fbdef} as before, but which is  not bounded below, and a function $\psi:\Omega\to\R$ satisfying \eqref{form} (with $\gamma=0$), where the condition \eqref{bedcond} is replaced by \begin{equation}\nabla\psi(X,Y)\to(0,c)\text{ as }Y\to-\infty.\tag{\underline{2.2b}}\label{infdepthcond}\end{equation}

For any $p\ge0$ and $\al\in(0,1)$, let $C^{p,\al}$ denote the space of functions whose partial derivatives of order up to $p$ are H\"older continuous of exponent $\al$ over their domain of definition. Let us denote by $C^{p,\al}_L$ the subspace of $C^{p,\al}(\R)$ consisting of functions that are $L$-periodic. Taking $L=2\pi$, we define the mean over one period of any function $f\in C^{p,\al}_{2\pi}$ by $$[f]:=\frac1{2\pi}\int_{-\pi}^\pi f(t)\,\mathrm{d}t.$$We denote by $C^{p,\al}_{2\pi,0}$ those functions $f\in C^{p,\al}_{2\pi}$ with zero mean, i.e. $[f]=0$.

Throughout this paper, our interest is in solutions $(\Omega,\psi)$ of the problem \eqref{form} of class $C^{2,\al}$, for some $\al\in(0,1)$, in the sense that the free boundary, $\cal{S}$, has a parameterisation \eqref{surface} with $(s\mapsto u(s)-s),v$ both functions in $C^{2,\al}_L$ with\begin{equation}u'(s)^2+v'(s)^2\ne0\quad\text{for all }s\in\R,\label{nostag}\end{equation}while $\psi\in C^\infty(\Omega)\cap C^{1,\al}(\overline{\Omega})$.
\subsection{A first reformulation}
Using the approach presented in \cite{CV}, it was shown in \cite{M} that \eqref{form} can be reformulated as a problem on a fixed domain, via the use of conformal mappings. The approach to this formulation is as follows.

For any $d>0$ let $$\cal{R}_d:=\{(x,y)\in\R^2:-d<y<0\}.$$Given any $w\in C^{p,\al}_{2\pi}$ let $W\in C^{p,\al}(\overline{\cal{R}_d})$ be the unique solution of \begin{align*}\Delta W&=0\quad\text{in }\cal{R}_d,\\W(x,-d)&=0,\quad x\in\R,\\W(x,0)&=w(x),\quad x\in\R.\end{align*}The function $(x,y)\mapsto W(x,y)$ is $2\pi$-periodic in $x$ throughout $\cal{R}_d$. Let $Z$ be the (up to a constant) unique harmonic conjugate of $W$, that is, such that $Z+iW$ is holomorphic on $\cal{R}_d$. As described thoroughly in \cite[\S2]{CV}, if $w\in C^{p,\al}_{2\pi,0}$ then the function $(x,y)\mapsto Z(x,y)$ is $2\pi$-periodic in $x$ throughout $\cal{R}_d$. We then prescribe the constant in the definition of $Z$ by requiring that the function $x\mapsto Z(x,0)$ has zero mean.

Then the operator $\cal{C}_d:C^{p,\al}_{2\pi,0}\to C^{p,\al}_{2\pi,0}$ is defined by $$\cal{C}_d(w)(x)=Z(x,0),\quad x\in\R.$$(It was shown in \cite{CV} that the operator $C_d$ satisfies some sort of Privalov's Theorem, and maps $C^{p,\al}$ functions into $C^{p,\al}$ indeed.)

It was also shown in \cite{CV} that for any $L$-periodic strip-like domain, $\Omega$, there is a unique positive number $h$, called its conformal mean depth, such that $\Omega$ is the image of a conformal mapping $U+iV$ from $\Rkh$ with \begin{eqnarray*}U(x+2\pi,y)&=&U(x,y)+L\\V(x+2\pi,y)&=&V(x,y)\end{eqnarray*}for $(x,y)\in\Rkh$, where $L=2\pi/k$. In particular, it follows that $\cal{S}$ can be parameterised by functions $u\in C^{2,\al},v\in C^{2,\al}_{2\pi}$ with mean $[v]=h$. Therefore, it follows that for any $L$-periodic strip-like domain we may make a vertical translation to place its flat bottom at $Y=-h$, where $h$ is its conformal mean depth, whence it is clear that the free surface will admit a parameterisation with a zero-mean $Y$ coordinate. This is especially important since we wish this theory to agree with the infinite depth problem, which, as will be seen later, corresponds formally to setting $\gamma=0$ and $h=\infty$.

Bearing in mind that we are free to make this vertical translation we now present the reformulation, as given in \cite[Theorem 1]{M}, but where suitable adjustments have been made to take into account this vertical translation, and also a substitution has been made to rescale the function of interest.

We have that the problem of finding $(\Omega,\psi)$ of class $C^{2,\al}$ solving \eqref{form} is equivalent to finding an $h>0$ and $w\in C^{2,\al}_{2\pi,0}$ such that \begin{subequations}\label{next}\begin{align}&\left\{\frac{m}{h}-\frac{\gamma h}{2}+\frac\gamma k\left(\frac{[w^2]}{2kh}+\Ckh(ww')-w-w\Ckh w'\right)\right\}^2\nonumber\\&\qquad\qquad\qquad\qquad\qquad=\left(Q-\frac{2gw}k+2\sigma k\frac{(1+\Ckh w')w''-w'\Ckh w''}{\Wkh(w)^{3/2}}\right)\Wkh(w)\label{2eqn}\\&w(t)>-kh\quad\text{for all }t\in\mathbb{R}&\label{abovebed}\\&\text{the mapping }t\mapsto\left(t+\Ckh w(t),w(t)\right)\text{ is injective on }\mathbb{R}\label{injcond}\\&\Wkh(w)(t)\ne0\quad\text{for all }t\in\mathbb{R},\label{Wkh0}\end{align}\end{subequations}where \begin{equation}\Wkh(w)=w'^2+(1+\Ckh w')^2\label{Wkh}\end{equation}and, as above, $L=2\pi/k$.

Here it should be noted that the formulation in \cite{M} has been slightly changed, using the general theory of periodic harmonic functions on a strip from \cite{CV}, to use the operator $\Ckh$, the periodic Hilbert transform for a strip.  A full and coherent treatment of the properties of this and related operators can be found in \cite{CV}.

Given an $h>0$ and $w\in C^{2,\al}_{2\pi,0}$ satisfying \eqref{next}, it then follows that the fluid region is bounded by the curves $$\mathcal{B}=\{(X,-h):X\in\mathbb{R}\},\qquad\mathcal{S}=\left\{\left(\frac1k(t+\Ckh w(t)),\frac{w(t)}k\right):t\in\mathbb{R}\right\}.$$

We remark at this point that it should be noted that for the purposes of the analysis we shall ignore the conditions \eqref{abovebed}, \eqref{injcond}, simply solving \eqref{2eqn} and \eqref{Wkh0} then discarding any solutions which violate \eqref{abovebed}, \eqref{injcond} afterwards. Hence, throughout we shall remove these conditions, leaving it to the reader to remember that any solutions found may have to be discarded if it happens that they violate one or both of these conditions.

In the case of infinite-depth waves (with, recall, zero vorticity), the problem (\underline{2.2}) can be reformulated in a similar manner, where, instead of viewing the fluid domain as the conformal image of a strip, we view $\Omega$ as the conformal image of the lower half plane.

Then, the problem of finding $(\Omega,\psi)$ of class $C^{2,\al}$ solving (\underline{2.2}) is equivalent to finding $w\in C^{2,\al}_{2\pi,0}$ such that\begin{subequations}\label{infnext}\begin{align}&c^2=\left(Q-\frac{2gw}k+2\sigma k\frac{(1+\cal{C}w')w''-w'\cal{C}w''}{\cal{W}(w)^{3/2}}\right)\cal{W}(w)\label{infeqn}\\&\text{the mapping }t\mapsto\left(t+\cal{C}w(t),w(t)\right)\text{ is injective on }\mathbb{R}&&&&&\label{infinjcond}\\&\cal{W}(w)(t)\ne0\quad\text{for all }t\in\mathbb{R},\label{W0}\end{align}\end{subequations}where \begin{equation}\cal{W}(w)=w'^2+(1+\cal{C} w')^2\label{Wdef}\end{equation}and, as above, $L=2\pi/k$. Here $\cal{C}$ denotes the usual periodic Hilbert transform.

As remarked before in the case of finite-depth, here we also disregard condition \eqref{infinjcond} and leave it to the reader to remember that any solutions found later may not be physical if it happens that they violate this condition.
\subsection{An additional reformulation of the problem}As derived above, we aim to find $h,k,m,Q,\gamma\in\mathbb{R},\;h,k>0$ and $w\in C^{2,\al}_{2\pi,0}$ such that \eqref{next} holds. However, it will be seen that this formulation will not be suitable when we come to the proof of our main theorem. For this we will require an alternative formulation, the proof of which can be found in \cite{deB}\begin{thm}For $\si\ne0$, and for any $\gamma\in\mathbb{R}$ and $h,k>0$, we have that $m,Q\in\mathbb{R}$ and $w\in C^{2,\al}_{2\pi,0}$ satisfy \eqref{2eqn}, subject to \eqref{Wkh0}, if and only if \begin{subequations}\label{2nd}\begin{multline}Q=\left[\Wkh(w)^{1/2}\right]^{-1}\left[\left\{\frac{m}{h}-\frac{\gamma h}{k}+\frac{\gamma}{k}\left(\frac{[w^2]}{2kh}+\mathcal{C}_{kh}(ww')-w-w\mathcal{C}_{kh}w'\right)\right\}^2\Wkh(w)^{-1/2}\right.
\\\left.+\frac{2gw}k\Wkh(w)^{1/2}\right]\label{mudef}\end{multline}and\begin{multline}w''=\frac{w'}{2\sigma k}\Ckh\left(\left\{\frac{m}{h}-\frac{\gamma h}2+\frac{\gamma}k\left(\frac{[w^2]}{2kh}+\mathcal{C}_{kh}(ww')-w-w\mathcal{C}_{kh}w'\right)\right\}^2\Wkh(w)^{-1/2}-\right.\\\left.-\left(Q-\frac{2gw}k\right)^{\phantom{2}}\!\!\Wkh(w)^{1/2}\right)+\\+\frac{1}{2\sigma k}(1+\Ckh w')\left(\left\{\frac{m}{h}-\frac{\gamma h}2+\frac{\gamma}k\left(\frac{[w^2]}{2kh}+\mathcal{C}_{kh}(ww')-w-w\mathcal{C}_{kh}w'\right)\right\}^2\Wkh(w)^{-1/2}\right.-\\\left.-\left(Q-\frac{2gw}k\right)^{\phantom{2}}\!\!\Wkh(w)^{1/2}\right),\label{2ndeqn}\end{multline}subject to \begin{equation}\Wkh(w)(t)\ne0\qquad\text{for all }t\in\mathbb{R}.\label{2ndWkh0}\end{equation}\end{subequations}That is, \eqref{next} is equivalent to \eqref{2nd}.\end{thm}In the case of infinite depth, the formulation \eqref{2nd} becomes, in the natural way, the problem of finding $w\in C^{2,\al}_{2\pi,0}$ such that \begin{subequations}\label{inf2nd}\begin{align}\nonumber&w''=\frac{w'}{2\sigma k}\cal{C}\!\left(c^2\cal{W}(w)^{-1/2}-\left(Q-\frac{2gw}k\right)\cal{W}(w)^{1/2}\right)+\\&\qquad\qquad+\frac{1}{2\sigma k}(1+\cal{C}w')\left(c^2\cal{W}(w)^{-1/2}-\left(Q-\frac{2gw}k\right)\cal{W}(w)^{1/2}\right)\label{inf2ndeqn}\\&\cal{W}(w)(t)\ne0\quad\text{for all }t\in\mathbb{R}.\label{2ndW0}\end{align}\end{subequations}That is, \eqref{infnext} is equivalent to \eqref{inf2nd}.

The advantage of appealing to these reformulations is that, as was shown in \cite{deB}, the right-hand side of \eqref{2ndeqn} (and also \eqref{inf2ndeqn}) can be viewed as a mapping from $C^{2,\al}$ to $C^{1,\al}$, since no more than one derivative is taken. However, $w''\in C^{0,\al}$ and so we must formally require this mapping to have $C^{0,\al}$ for its codomain, into which space $C^{1,\al}$ embeds compactly, thus demonstrating it to be a compact mapping. Later, this will allow us to establish that certain related linear operators are Fredholm operators of index $0$, allowing a certain simplification of the calculations undertaken to show bijectivity of said operators.
\section{Crapper's Waves}\label{crapsol}In this section, we shall show that $(\be_A,w_A)\in\R_+\times C^{2,\delta}_{2\pi,0,e}$, where $$\be_A:=\frac{1+A^2}{1-A^2},$$ is a solution to \eqref{deepc} -- or, indeed \eqref{infeqn} -- for any $A\in(-1,1)$, where we have made the change to $(\al,\be)$-variables and we take $Q/c^2=1$.

However, before we do so, we must show that $w_A\in C^{2,\delta}_{2\pi,0,e}$ indeed. Clearly, the function is $2\pi$-periodic and even; it is also trivial to note that since $|A|<1$ the function is infinitely differentiable.  It thus remains to show that $[w_A]=0$. For any $A\in(-1,1)$, consider the function given by $$F_A(z)=\frac{2(1-Az)}{1+Az}-2.$$Since it can be observed that this function is complex-differentiable away from its poles, of which it has one at $z=-1/A$ (which is outside the unit disc), it follows that this function is analytic on the unit disc. Then, since $F_A(0)\in\R$, it can be seen that \begin{equation}F_A(e^{it})=\frac{2(1-A^2)}{1+A^2+2A\cos t}-2-i\frac{4A\sin t}{1+A^2+2A\cos t}=w_A(t)+i\cal{C}w_A(t).\label{FAdef}\end{equation}It therefore follows that $$[w_A]:=\frac1{2\pi}\int_{-\pi}^\pi w_A(t)\,\mathrm{d}t=\re\left(\frac1{2\pi}\int_{-\pi}^\pi F_A(e^{it})\,\mathrm{d}t\right)=\re F_A(0)=0,$$where we have used the mean-value property for harmonic functions. Hence $w_A\in C^{2,\delta}_{2\pi,0,e}$ indeed.

It thus remains to show that $w_A$ satisfies $$\left((1+\cal{C}w_A')^2+w_A'^2\right)^{1/2}+2\be_A\frac{(1+\cal{C}w_A')w_A''-w_A'\cal{C}w_A''}{(1+\cal{C}w_A')^2+w_A'^2}=\left((1+\cal{C}w_A')^2+w_A'^2\right)^{-1/2}$$for $A\in(-1,1)$ and $\be_A=(1+A^2)/(1-A^2)$.

Recalling \eqref{FAdef}, we have that $$w_A'(t)+i(1+\cal{C}w_A')(t)=i+\frac{\mathrm{d}}{\mathrm{d}t}F_A(e^{it})=i+izF_A'(z)\big|_{z=e^{it}}$$and$$w_A''(t)+i\cal{C}w_A''(t)=\frac{\mathrm{d}^2}{\mathrm{d}t^2}F_A(e^{it})=\big(-z^2F_A''(z)-zF_A'(z)\big)\big|_{z=e^{it}}.$$For $w\in C^{2,\delta}_{2\pi}$, it is easy to verify the functional identity $$\frac{(1+\cal{C}w')w''-w'\cal{C}w''}{(1+\cal{C}w')^2+w'^2}=-\imag\left(\frac{w''+i\cal{C}w''}{w'+i(1+\cal{C}w')}\right).$$This implies that $$\frac{(1+\cal{C}w_A')w_A''-w_A'\cal{C}w_A''}{(1+\cal{C}w_A')^2+w_A'^2}=\left.\imag\left(\frac{-z^2F_A''(z)-zF_A'(z)}{i+izF_A'(z)}\right)\right|_{z=e^{it}}.$$Note that $$i+izF_A'(z)=\frac{i(1-Az)^2}{(1+Az)^2};\qquad z^2F_A''(z)+zF_A'(z)=-\frac{4Az(1-Az)}{(1+Az)^3}.$$Now, from the definition of $F_A$ and previous computations we have that $$\frac{z^2F_A''(z)+zF_A'(z)}{i+izF_A'(z)}=\frac{4Aiz}{(1+Az)(1-Az)}.$$We consider $$\frac{1}{2i}\left(\frac{4Aiz}{(1+Az)(1-Az)}+\frac{4Ai\bar z}{(1+A\bar z)(1-A\bar z)}\right)=\frac{2A(z+\bar z)(1-A^2|z|^2)}{|1+Az|^2|1-Az|^2}.$$Combining this with calculations from above, we obtain that \begin{eqnarray*}\imag\left(\frac{z^2F_A''(z)+zF_A'(z)}{i+izF_A'(z)}\right)&=&\frac{4A(1-A^2)\cos t}{|1+Ae^{it}|^2|1-Ae^{it}|^2}\qquad\text{at }z=e^{it}.\end{eqnarray*}In view of the calculations above, it can be seen that we aim to show that$$\frac{|1-Ae^{it}|^2}{|1+Ae^{it}|^2}+\frac{8A(1+A^2)\cos{t}}{|1+Ae^{it}|^2|1-Ae^{it}|^2}=\frac{|1+Ae^{it}|^2}{|1-Ae^{it}|^2},$$which is a simple calculation. 
\section{A sheet of solutions into which Crapper's waves are embedded.}\label{crapbif} In this section, we shall see that there exist physically-relevant solutions to the full infinite-depth gravity-capillary problem that are `close' in some sense to Crapper's waves.  In particular, it will be shown that, under a given scaling regime, the limiting form of the governing equation is exactly the equation satisfied by Crapper's waves, and this will be used to find solutions with limiting profiles under the aforementioned scaling given exactly by a Crapper's wave.

Then, our approach will be essentially the same, except that we shall be considering finite-depth flows that are both irrotational and possessing a non-zero constant vorticity. Specifically, we shall demonstrate that under the same scaling regime we shall obtain that the governing equation -- even in this more general scenario -- has the same limiting form; that is, the infinite-depth pure-capillary equations. An application of the Implicit Function Theorem then yields solutions, exactly as before, to the full problem with Crapper's waves as their limiting profile.

As might have been expected, the calculations are slightly more delicate and so more care must be taken. 

Both of these results will follow (after some preliminary calculations) directly from the Implicit Function Theorem, which we state here, for completeness. We use the form of this Theorem as given by Kielh\"ofer in \cite{K}, but this is a completely standard Theorem that can be found in many works, e.g. \cite[Theorem 4.B]{Z}.

In particular, note that this Theorem requires no greater regularity of the whole mapping than continuity. That is to say, whilst the mapping must clearly be continuously differentiable with respect to the implicit variable, it is \emph{not} required with respect to (any) other variables, with respect to which it is simply required to be continuous.
\begin{thm}Let $X,Y,Z$ be (real) Banach spaces, and let $F:U\times V\to Z$ be a continuous mapping, with $U\subset X,V\subset Y$ open sets. Suppose that $\partial_yF$ is continuous on $U\times V$ and let $(x_0,y_0)\in U\times V$ such that $F(x_0,y_0)=0$. Suppose that $\p_yF[x_0,y_0]$ is an isomorphism; that is, $\p_yF[x_0,y_0]:Y\to Z$ is bounded with a bounded inverse. Then there exists a neighbourhood $U_1\times V_1\subset U\times V$ of $(x_0,y_0)$ and a continuous mapping $f:U_1\to V_1$ with $f(x_0)=y_0$ such that $$F(x,f(x))=0\quad\text{for all }x\in U_1$$ and, further, that every solution of $F(x,y)=0$ for $(x,y)\in U_1\times V_1$ is of the form $(x,f(x))$.\end{thm}
Let \begin{multline}\F(\al,\be,w):=w''-\frac{w'}{2\be}\cal{C}\left(\cal{W}(w)^{-1/2}-(b(\al,w)-2\al w)\cal{W}(w)^{1/2}\right)\\-\frac{1}{2\be}(1+\cal{C}w')\left(\cal{W}(w)^{-1/2}-(b(\al,w)-2\al w)\cal{W}(w)^{1/2}\right),\label{crapFdef}\end{multline}where,\begin{equation}b(\al,w):=\left[\cal{W}(w)^{1/2}\right]^{-1}\left[\cal{W}(w)^{-1/2}+2\al w\cal{W}(w)^{1/2}\right]\label{fullbdef}\end{equation} is defined such that the right-hand side of \eqref{crapFdef} has mean 1, and, as before,  $$\cal{W}(w)=w'^2+(1+\cal{C}w')^2.$$In view of \S2, it follows that $\F(\al,\be,w)=0$ if and only if $(\al,\be,w)$ satisfy \eqref{infeqn} after the change to $(\al,\be)$-variables. It is then clear that, in order to find solutions to the gravity-capillary wave problem, we are searching for $(\al,\be,w)\in\R_+\times\R_+\times X$ such that $\F(\al,\be,w)=0$. However, there is no mathematical reason to restrict to $\al>0$. Indeed, in \S\ref{crapsol}, we saw that Crapper's waves, given by $(0,\be_A,w_A)$ for $A\in(-1,1)$, are solutions to \eqref{infeqn} and thus, via the equivalence above, to $\F=0$.

Observe now that, in view of the comments at the end of \S2, and the Fredholm alternative, it follows that $\p_w\F[\al,\be,w]$ is a Fredholm operator of index $0$ whenever $\F(\al,\be,w)=0$. In particular, note that this implies that $\p_w\F[0,\be_A,w_A]$ is bijective if and only if it is injective.

Now let us consider the case of finite-depth and constant vorticity. We begin with \eqref{2ndeqn}, where we have made the substitution $$\la=\frac{m}{h}-\frac{\gamma h}2.$$Thus, we have \begin{multline}w''=\frac{w'}{2\sigma k}\Ckh\left(\left\{\la+\frac{\gamma}k\left(\frac{[w^2]}{2kh}+\mathcal{C}_{kh}(ww')-w-w\mathcal{C}_{kh}w'\right)\right\}^2\Wkh(w)^{-1/2}-\right.\\\left.-\left(Q(\la,k,w)-\frac{2gw}k\right)\Wkh(w)^{1/2}\right)+\\+\frac{1}{2\sigma k}(1+\Ckh w')\left(\left\{\la+\frac{\gamma}{k}\left(\frac{[w^2]}{2kh}+\mathcal{C}_{kh}(ww')-w-w\mathcal{C}_{kh}w'\right)\right\}^2\Wkh(w)^{-1/2}\right.\\\left.-\left(Q(\la,k,w)-\frac{2gw}k\right)\Wkh(w)^{1/2}\right),\label{fdeqn}\end{multline}where $Q(\la,k,w)$ is defined so that the right-hand side of \eqref{fdeqn} has zero mean.  Noting that the case for $\la<0$ can be treated entirely similarly, we consider the following bijective change of variables, exactly as before, $$(\al,\be)=\left(\frac{g}{k\la^2},\frac{k\si}{\la^2}\right),$$where, once again, we consider $g,\si$ to be fixed positive constants. With these new variables, \eqref{fdeqn} now becomes \begin{equation}w''=\frac{w'}{2\be}\cal{C}_{hk(\al,\be)}\left(\cal{A}(\al,\be,w)\right)+\frac{1}{2\be}\left(1+\cal{C}_{hk(\al,\be)}w'\right)\cal{A}(\al,\be,w),\label{resceqn}\end{equation}where \begin{multline*}\cal{A}(\al,\be,w):=\\\left\{1+\gamma\sqrt[4]{\frac{\al^3\si}{g^3\be}}\left(\frac{[w^2]}{2h}\sqrt{\frac{\al\si}{g\be}}+\cal{C}_{hk(\al,\be)}(ww')-w-w\cal{C}_{hk(\al,\be)}w'\right)\right\}^2\cal{W}_{hk(\al,\be)}(w)^{-1/2}-\\-\left(\hat Q(\al,\be,w)-2\al w\right)\cal{W}_{hk(\al,\be)}\end{multline*}(with $\hat Q$ given such that $\hat Q(g/k\la^2,k\si/\la^2,w)=Q(\la,k,w)/\la^2$) and $$k(\al,\be)=\sqrt{\frac{g\be}{\al\si}}.$$It should be clear that, as above, we wish to consider the limiting form of this equation as $\al\to0$. However, it is obvious that, in contrast to the procedure employed previously, the equation \eqref{resceqn} is not mathematically well-defined for $\al=0$. In particular, we observe that $k(\al,\be)\to\infty$ as $\al\to0$, meaning that the operators $\cal{C}_{hk(\al,\be)}$ and $\cal{W}_{hk(\al,\be)}$ are undefined for $\al=0$.  Note also that neither $\cal{A}$ nor $k$ are well-defined for $\al<0$ too.

Nonetheless, it can be seen from \cite{CV}, that it should be expected that $\cal{C}_d\to\cal{C}$ as $d\to\infty$ in a rigorous and well-defined manner. Since we aim to apply the Implicit Function Theorem, which, as shown above demands only continuity of the mapping and not necessarily continuous differentiability, it will suffice to exploit this property of $\cal{C}_d$ to construct a continuous mapping as desired.

Indeed, we have the following Lemma.\begin{lem}Given\label{Cdwcont} any $w_0\in C^{p,\delta}_{2\pi,0}$ and a sequence $((d_n,w_n))\subset\R_+\times C^{p,\delta}_{2\pi,0}$ such that $(d_n,w_n)\to(\infty,w_0)$ as $n\to\infty$, it holds that $$\cal{C}_{d_n}w_n\to\cal{C}w_0$$in $C^{p,\delta}_{2\pi,0}$.\end{lem}\begin{proof}Recall the natural embeddings of the H\"older spaces in the Sobolev spaces of weakly-differentiable square-integrable functions, that is $$H^{p+1}\subset C^{p,\delta}\subset H^p.$$ We shall employ this to appeal to the theory presented in \cite{CV}.

It is clear that we wish to show that $$\|\cal{C}_{d_n}w_n-\cal{C}w_0\|_{C^{p,\delta}}\to0$$as $n\to\infty$. Observing that $$\|\cal{C}_{d_n}w_n-\cal{C}w_0\|_{C^{p,\delta}}\le\|\cal{C}_{d_n}w_n-\cal{C}w_n\|_{C^{p,\delta}}+\|\cal{C}w_n-\cal{C}w_0\|_{C^{p,\delta}},$$and noting that the second term on the right-hand side vanishes as $n\to\infty$, by dint of the boundedness of the operator $\cal{C}$ (via Privalov's Theorem \cite{CV}) and the convergence of $w_n$ to $w_0$, it is clear that it suffices to show that $$\|\cal{C}_{d_n}w_n-\cal{C}w_n\|_{C^{p,\delta}}\to0$$as $n\to\infty$. Now, $$\|\cal{C}_{d_n}w_n-\cal{C}w_n\|_{C^{p,\delta}}\le C\|\cal{C}_{d_n}w_n-\cal{C}w_n\|_{H^{p+1}},$$for some constant $C$, by the above embeddings. Then, as described in \cite{CV}, we have that $$\|\cal{C}_{d_n}w_n-\cal{C}w_n\|_{H^{p+1}}=\|\kappa_{d_n}*w_n\|_{H^{p+1}},$$which is finite, due to the regularising properties of $\kappa_d$, where\begin{equation}\kappa_d(t):=\sum_{m\in\mathbb{Z}\setminus\{0\}}\frac{-2i\sgn(m)}{e^{2|m|d}-1}e^{imt}=\sum_{m\in\mathbb{Z}\setminus\{0\}}-i\sgn(m)\la_me^{imt}\label{kappadef}\end{equation}for $t\in\R$. But, for $$w(t)=\sum_{m\in\mathbb{Z}\setminus\{0\}}a_me^{imt}\in H^{p}$$(which is $2\pi$-periodic and has zero mean), it follows that $$\|\kappa_d*w\|_{H^{p+1}}^2=\sum_{m\in\mathbb{Z}\setminus\{0\}}m^{2(p+1)}\la_m^2a_m^2=\sum_{m\in\mathbb{Z}\setminus\{0\}}(m^2\la_m^2)(m^{2p}a_m^2)\le\left(\sup_{m\in\mathbb{Z}\setminus\{0\}}|m\la_m|\right)^2\|w\|_{H^p}^2.$$Thus, we obtain that $$\|\kappa_{d_n}*w_n\|_{H^{p+1}}\le\|\kappa_{d_n}\|_{W^{1,\infty}}\|w_n\|_{H^p}\le C\|\kappa_{d_n}\|_{W^{1,\infty}}\|w_n\|_{C^{p,\delta}},$$for some constant $C$. Since $(w_n)$ converges in $C^{p,\delta}$, we have that $\|w_n\|_{C^{p,\delta}}$ is bounded uniformly in $n$, and so it follows that $$\|\cal{C}_{d_n}w_n-\cal{C}w_n\|_{C^{p,\delta}}\le C\|\kappa_{d_n}\|_{W^{1,\infty}},$$where $C$ is some constant that does not depend on $n$ or $d_n$. But then we observe from \eqref{kappadef} that $\la_m\to0$ as $d\to\infty$ for all $m\in\mathbb{Z}\setminus\{0\}$ and so necessarily $\|\kappa_{d_n}\|_{W^{1,\infty}}\to0$ as $n\to\infty$, thus concluding the proof.\end{proof}

Therefore, we consider the following linear operator $\mathfrak{C}_{\al,\be}:C^{p,\delta}_{2\pi}\to C^{p,\delta}_{2\pi}$ given by $$\mathfrak{C}_{\al,\be}:=\begin{cases}\cal{C}_{hk(\al,\be)}&\al>0\\\cal{C}&\al\le0\end{cases},$$which, in view of Lemma \ref{Cdwcont}, is continuous for all $\al\in\R$ and $\be>0$. Defining $$\mathfrak{W}_{\al,\be}(w):=w'^2+(1+\mathfrak{C}_{\al,\be}w')^2,$$ we can thus construct the mapping $\mathfrak{A}(\al,\be,w)$ for $\al\ge0$, $\be>0$ and $w\in C^{2,\delta}_{2\pi,0,e}$ via \begin{multline*}\mathfrak{A}(\al,\be,w):=\\\left\{1+\gamma\sqrt[4]{\frac{\al^3\si}{g^3\be}}\left(\frac{[w^2]}{2h}\sqrt{\frac{\al\si}{g\be}}+\mathfrak{C}_{\al,\be}(ww')-w-w\mathfrak{C}_{\al,\be}w'\right)\right\}^2\mathfrak{W}_{\al,\be}(w)^{-1/2}-\\-\left(\hat Q(\al,\be,w)-2\al w\right)\mathfrak{W}_{\al,\be},\end{multline*}where $\hat Q(0,\be,w):=b(0,w)$, which, as seen above, is readily calculable to be its limit as $\al\to0$.

Continuously extending $\mathfrak{A}(\al,\be,w):=\mathfrak{A}(0,\be,w)$ for $\al<0$,\footnote{We must define all of our mappings for at least some small interval of negative values of $\al$. This is because the Implicit Function Theorem requires values of $\al$ in an open set, and we demand inclusion of $\al=0$.} we can thus define a mapping $\mathfrak{F}:\R\times\R_+\times C^{2,\delta}_{2\pi,0,e}\to C^{0,\delta}_{2\pi,0,e}$ by \begin{equation}\mathfrak{F}(\al,\be,w)=w''-\frac{w'}{2\be}\mathfrak{C}_{\al,\be}\left(\mathfrak{A}(\al,\be,w)\right)-\frac{1}{2\be}\left(1+\mathfrak{C}_{\al,\be}w'\right)\mathfrak{A}(\al,\be,w),\label{frakFdef}\end{equation}whence it is simple to observe that $\mathfrak{F}(\al,\be,w)=0$ for $\al,\be>0$ if and only if \eqref{resceqn}, and thus \eqref{fdeqn} (after a change of parameters), holds.

However, now note that $\mathfrak{F}(0,\be,w)=\F(0,\be,w)$, as above. In particular, we observe that $$\p_w\mathfrak{F}[0,\be_0,w_0]=\p_w\F[0,\be_0,w_0]$$for any $(\be_0,w_0)\in\R_+\times C^{2,\delta}_{2\pi,0,e}$ and that $\mathfrak{F}(0,\be_A,w_A)=0$ for every $A\in(-1,1)$.

The aim, then, in this section is to prove the following theorems, which show that Crapper's waves are indeed the limiting form of real gravity-capillary waves.\begin{thm}For every $0\ne A\in(-1,1)$ there exists a neighbourhood $U_A$ of $(0,\be_A)$ in $\R\times\R_+$, a neighbourhood $V_A$ of $(0,\be_A,w_A)$ in $\mathbb{R}\times\mathbb{R}_+\times X$ and a function $W_A:U_A\to X$ such that $$\{(\al,\be,w)\in V_A:\F(\al,\be,w)=0\}=\{(\al,\be,W_A(\al,\be)):(\al,\be)\in U_A\}.$$In particular, each solution with $\al=0$ is of the form $(0,\be_B,w_B)$ for some $B\in(-1,1)$.\end{thm}
\begin{thm}For every $0\ne A\in(-1,1)$ there exists a neighbourhood $U_A$ of $(0,\be_A)$ in $\R\times\R_+$, a neighbourhood $V_A$ of $(0,\be_A,w_A)$ in $\mathbb{R}\times\mathbb{R}_+\times X$ and a function $W_A:U_A\to X$ such $$\{(\al,\be,w)\in V_A:\mathfrak{F}(\al,\be,w)=0\}=\{(\al,\be,W_A(\al,\be)):(\al,\be)\in U_A\}.$$In particular, each solution with $\al=0$ is of the form $(0,\be_B,w_B)$ for some $B\in(-1,1)$.\end{thm}
\begin{rems}It should be clear that, as a result of this theorem, we have found solutions to $\mathfrak{F}(\al,\be,w)$ with $\al>0$, which can be seen to approach Crapper's waves as $\al\to0$. That is to say, given the physical constants $g,\si>0$, for any $h,k,\la>0$ and $\gamma\in\R$ there exist gravity-capillary flows of conformal mean depth $h$, period $2\pi/k$, constant vorticity $\gamma$ and mass flux $$m=h\la+\frac{h^2\gamma}{2}.$$\item The second part of the first theorem proves, for our formulation, the result from \cite{OS}, which states that there is no secondary bifurcation from Crapper's waves in the pure capillary regime. In fact, what we shall prove is that, for $A\ne0$, in a neighbourhood of $(0,\be_A,w_A)$ all solutions of the form $(0,\be,w)$ to $\F=0$ or $\mathfrak{F}=0$ lie on a curve parameterised by $\be$.   It is easy to see that $A\mapsto\be_A$ is invertible for $A\ne0$ and so the result of no bifurcation from Crapper's waves is then immediate, since the curve of explicit solutions can be shown to be parameterised by $\be$ and thus is the curve found analytically.\end{rems}

We now make the following observation. Define a mapping $\Theta:X\to C^{1,\alpha}_{2\pi,0,o}$ by $\Theta(w)=\arg(1+\mathcal{C}w'+iw')$. Then it can be seen that $1+\mathcal{C}w'+iw'=\mathcal{W}(w)^{1/2}\exp(i\Theta(w))$. Notice now that $\mathcal{W}(w)^{1/2}\exp(i\Theta(w))=(1+\mathcal{C}w')+iw'$ is the boundary value of some function which is analytic on the lower half plane, which we require to be away from zero. It can also be clearly seen that it has values which tend to 1 as $\im z\to-\infty$. Then, by taking logarithms we see that $\log\mathcal{W}(w)^{1/2}+i\Theta(w)$ is the boundary value of some function which is analytic on the lower half plane, with its values tending to 0 as $\im z\to-\infty$; it then follows that $\mathcal{W}(w)^{1/2}=\exp(\mathcal{C}\Theta(w))$.

The fact that $\Theta$ is invertible follows naturally from its definition, and it can also be seen that $\mathrm{d}\Theta[w_A]$ is a homeomorphism.

We can then observe an equivalence between the formulation in terms of $w$ and the formulation in \cite{OS}, given in terms of $\theta$, as follows. \begin{lem}Let \label{FGequiv}$\be\in\R$ and define $\tilde F(w):=\F(0,\be,w)$, that is, \begin{multline}\tilde F(w):=\F(0,\be,w)=w''-\frac{w'}{2\be}\mathcal{C}\left(\mathcal{W}(w)^{-1/2}-\frac{\left[\cal{W}(w)^{-1/2}\right]}{\left[\cal{W}(w)^{1/2}\right]}\cal{W}(w)^{1/2}\right)-\\-\frac{1}{2\be}(1+\mathcal{C}w')\left(\mathcal{W}(w)^{-1/2}-\frac{\left[\cal{W}(w)^{-1/2}\right]}{\left[\cal{W}(w)^{1/2}\right]}\cal{W}(w)^{1/2}\right);\label{tildeFdef}\end{multline}and define \begin{equation}G(\theta):=\theta'-\frac{1}{2\be}\exp(-\cal{C}\theta)+\frac{1}{2\be}\frac{\left[\exp(-\cal{C}\theta)\right]}{\left[\exp(\cal{C}\theta)\right]}\exp(\cal{C}\theta).\label{Gdef}\end{equation}Then $\tilde F(w)=0$ if and only if $G(\Theta(w))=0$.\end{lem}\begin{rem}In \cite{OS}, it was shown that when considering only infinite-depth pure-capillary waves there is no loss of generality in assuming (in our notation) that $b\equiv1$. Thus, they consider the equation \begin{equation}\theta'+q\sinh{\cal{C}\theta}=0,\label{nob}\end{equation}for a parameter $q\in\R$ corresponding to $1/\be$. Since we consider this regime as a limiting one, we have made a choice for the form of $b$ that suits us better, but is not necessarily identically $1$ in the limit, and thus it cannot be removed from the definition. Whilst \eqref{Gdef} is inelegant, and does not illustrate its similarity and close connection with \eqref{nob}, it allows the casual reader to understand the inherent structure more clearly.\end{rem}\begin{proof}Noting that $$w'=\mathcal{W}(w)^{1/2}\sin\Theta(w)=\exp(\mathcal{C}\Theta(w))\sin\Theta(w)\quad\text{ and }\quad1+\mathcal{C}w'=\exp(\mathcal{C}\Theta(w))\cos\Theta(w),$$ it is clear that $\tilde F(w)=\tilde G(\Theta(w))$, where \begin{multline}\tilde G(\theta)=(\exp(\mathcal{C}\theta)\sin\theta)'-\frac{1}{2\be}\exp(\mathcal{C}\theta)\sin\theta\,\mathcal{C}\!\left(\exp(-\mathcal{C}\theta)\right)+\frac{1}{2\be}\frac{\left[\exp(-\cal{C}\theta)\right]}{\left[\exp(\cal{C}\theta)\right]}\exp(\cal{C}\theta)\sin\theta\cal{C}\!\left(\exp(\cal{C}\theta)\right)\\-\frac{1}{2\be}\exp(\mathcal{C}\theta)\cos\theta\exp(-\mathcal{C}\theta)+\frac{1}{2\be}\frac{\left[\exp(-\cal{C}\theta)\right]}{\left[\exp(\cal{C}\theta)\right]}\exp(\cal{C}\theta)\cos\theta\exp(\cal{C}\theta).\label{tildeGdef}\end{multline}First observe that using the product rule and the linearity of $\mathcal{C}$ we can rewrite $\tilde G$ as \begin{eqnarray*}\tilde G(\theta)&=&\exp(\mathcal{C}\theta)\sin\theta\,\mathcal{C}\!\!\left(\theta'-\frac{1}{2\be}\exp(-\mathcal{C}\theta)+\frac{1}{2\be}\frac{\left[\exp(-\cal{C}\theta)\right]}{\left[\exp(\cal{C}\theta)\right]}\exp(\cal{C}\theta)\right)\\&&+\exp(\mathcal{C}\theta)\cos\theta\left(\theta'-\frac{1}{2\be}\exp(-\mathcal{C}\theta)+\frac{1}{2\be}\frac{\left[\exp(-\cal{C}\theta)\right]}{\left[\exp(\cal{C}\theta)\right]}\exp(\cal{C}\theta)\right)\\&=&\exp(\mathcal{C}\theta)\sin\theta\,\mathcal{C}(G(\theta))+\exp(\mathcal{C}\theta)\cos\theta\,G(\theta).\end{eqnarray*}We claim that this gives $G(\theta)=0$ if and only if $\tilde G(\theta)=0$; the proof then following immediately.

Indeed, let $A,B$ be functions analytic on the lower half-plane with boundary values $$\exp(\cal{C}\theta)\cos\theta+i\exp(\cal{C}\theta)\sin\theta,\qquad\cal{C}(G(\theta))+iG(\theta)$$respectively, both of which having values that tend to 0 as $\imag z\to-\infty$. It is then clear that the function $AB$ is analytic on the lower half-plane, also with values tending to 0 as $\imag z\to-\infty$, and such that the imaginary part of its boundary values is exactly $\tilde G(\theta)$. It then follows that the boundary values of $AB$ are given by $\cal{C}(\tilde G(\theta))+i\tilde G(\theta)$.  Since it is clear that the boundary values of $A$ are never zero, we have that the boundary values of $AB$ are zero if and only if the boundary values of $B$ are zero. This then naturally yields the claim.\end{proof}
Now, let us explain our motivation for introducing this formulation involving $\theta$.  Later, when proving the main theorems, we shall require the fact that $\p_w\F[0,\be_A,w_A]$ is an injective operator. However, this fact is not immediately obvious.  Nonetheless, we shall see in the following lemma that it will be equivalent to consider the linearisation of another mapping.
\begin{lem}For\label{FinjGinj} any $(\beta_0,w_0)\in\R_+\times C^{2,\delta}_{2\pi,0,e}$ such that $\F(0,\be_0,w_0)=0$ we have that $\p_w\F[0,\be_0,w_0]$is an injective linear operator if and only if $\mathrm{d}G[\Theta(w_0)]$ is an injective linear operator (where $\be=\be_0$ is taken in the definition of $G$ above).\end{lem}\begin{proof}It is clear that $$\p_w\F[0,\be_0,w_0]=\mathrm{d}\tilde F[w_0],$$where $\be=\be_0$ is taken in the definition of $\tilde F$ above, and so we shall focus our attention upon $\mathrm{d}\tilde F$, since it will be enough to prove that $\mathrm{d}\tilde F[w_0]$ is injective.

But then we observe that $$\mathrm{d}\tilde F[w_0]w=\mathrm{d}\tilde G[\Theta(w_0)]\circ\mathrm{d}\Theta[w_0]w=\mathrm{d}\tilde G[\theta_0]\circ\mathrm{d}\Theta[w_0]w,$$where $\theta_0=\Theta(w_0)$. Since $\mathrm{d}\Theta[w_0]$ is a homeomorphism, the injectivity of $\mathrm{d}\tilde F[w_0]$ will follow if we show that $\mathrm{d}\tilde G[\theta_0]$ is injective. Therefore, we begin by looking at $\mathrm{d}\tilde G[\theta_0]$. We now linearise $\tilde G$ about $\theta_0$.  We get \begin{multline*}\mathrm{d}\tilde G[\theta_0]\theta=\mathcal{C}\theta\exp(\mathcal{C}\theta_0)\sin\theta_0\,\mathcal{C}(G(\theta_0))+\mathcal{C}\theta\exp(\mathcal{C}\theta_0)\cos\theta_0\,G(\theta_0)+\theta\exp(\mathcal{C}\theta_0)\cos\theta_0\,\mathcal{C}(G(\theta_0))\\-\theta\exp(\mathcal{C}\theta_0)\sin\theta_0\,G(\theta_0)+\exp(\mathcal{C}\theta_0)\sin\theta_0\,\mathcal{C}(\mathrm{d}G[\theta_0]\theta)+\exp(\mathcal{C}\theta_0)\cos\theta_0\,\mathrm{d}G[\theta_0]\theta.\end{multline*}Observe that $\theta_0$ is a solution of $\tilde G(\theta)=0$, since $\tilde G(\theta_0)=\tilde G(\Theta(w_0))=\tilde F(w_0)=0$, by assumption. Therefore, we have from Lemma \ref{FGequiv} that $G(\theta_0)=0$, and so $$\mathrm{d}\tilde G[\theta_0]\theta=\exp(\mathcal{C}\theta_0)\sin\theta_0\,\mathcal{C}(\mathrm{d}G[\theta_0]\theta)+\exp(\mathcal{C}\theta_0)\cos\theta_0\,\mathrm{d}G[\theta_0]\theta;$$by an entirely similar argument to that given in the Lemma above, it then follows that $\mathrm{d}\tilde G[\theta_0]\theta=0$ if and only if $\mathrm{d}G[\theta_0]\theta=0$. The lemma then trivially follows.\end{proof}

It is therefore clear that we should wish to prove the following result.
\begin{lem}For\label{Ginj} every $0\ne A\in(-1,1)$ the linear operator $\mathrm{d}G[\theta_A]$ is injective, where $\theta_A:=\Theta(w_A)$ and $\be=\be_A$ is taken in the definition of $G$.\end{lem}
\begin{rem}This lemma is, in essence, the condition proven in \cite{OS} in order to conclude that there is no secondary bifurcation from Crapper's waves.\end{rem}
\begin{proof}Following a similar argument to that given in \cite{OS} -- which we have adapted for our particular situation -- we can show that $\mathrm{d}G[\theta_A]$ is injective and hence conclude the injectivity of $\mathrm{d}\tilde G[\theta_A]$, $\mathrm{d}\tilde F[w_A]$ and, principally, $\p_w\F[0,\be_A,w_A]$. First, we calculate that $$\mathrm{d}G[\theta_A]\theta=\theta'+\frac{1}{2\be_A}\exp(-\mathcal{C}\theta_A)\mathcal{C}\theta+\frac{1}{2\be_A}\frac{\left[\exp(-\cal{C}\theta_A)\right]}{\left[\exp(\cal{C}\theta_A)\right]}\exp(\cal{C}\theta_A)\cal{C}\theta+C_A(\theta)\exp(\cal{C}\theta_A),$$for some appropriate constant $C_A(\theta)$.\footnote{An explicit calculation shows that $C_A(\theta)$ is exactly the value that would be found by solving $[\mathrm{d}G[\theta_A]\theta]=0$.} It can be calculated that $\left[\exp(-\cal{C}\theta_A)\right]=\left[\exp(\cal{C}\theta_A)\right]$, whence we obtain that $$\mathrm{d}G[\theta_A]\theta=\theta'+\frac{1}{2\be_A}\exp(-\cal{C}\theta_A)\cal{C}\theta+\frac{1}{2\be_A}\exp(\cal{C}\theta_A)\cal{C}\theta+C_A(\theta)\exp(\cal{C}\theta_A).$$ Then, by explicitly calculating $\exp(\pm\mathcal{C}\theta_A)$, setting $q_A=1/\be_A=(1-A^2)/(1+A^2)$ and writing $$\theta(t)=\sum_{n=1}^\infty a_n\sin{nt}\qquad\text{for }\theta\in C^{1,\delta}_{2\pi,0,o}$$ we have that $\mathrm{d}G[\theta_A]\theta=0$ is equivalent to \begin{multline*}\sum_{n=1}^\infty na_n\cos{nt}-\frac{q_A}{2}\frac{1+A^2+2A\cos{t}}{1+A^2-2A\cos{t}}\sum_{n=1}^\infty a_n\cos{nt}-\frac{q_A}{2}\frac{1+A^2-2A\cos{t}}{1+A^2+2A\cos{t}}\sum_{n=1}^\infty a_n\cos{nt}+\\+C_A(\theta)\frac{1+A^2-2A\cos{t}}{1+A^2+2A\cos{t}}=0.\end{multline*}Multiplying both sides by the common denominator of the left-hand side, we have that this is equivalent to \begin{multline*}(1+A^4)\sum_{n=1}^\infty(n-q_A)a_n\cos{nt}-A^2\sum_{n=1}^\infty(n+q_A)a_n\big(\cos{(n+2)t}+\cos{(n-2)t}\big)-4A^2q_A\sum_{n=1}^\infty a_n\cos{nt}\\+C_A(\theta)(1+4A^2+A^4-4A(1+A^2)\cos{t}+2A^2\cos{2t})=0.\end{multline*} Comparing coefficients, we have \begin{eqnarray}C_A(\theta)(1+4A^2+A^4)&=&A^2(2+q_A)a_2,\label{0coeff}\\((1+A^4)(1-q_A)-4A^2q_A)a_1&=&A^2(3+q_A)a_3+A^2(1+q_A)a_1+4A(1+A^2)C_A(\theta),\label{1coeff}\\((1+A^4)(2-q_A)-4A^2q_A)a_2&=&A^2(4+q_A)a_4-2A^2C_A(\theta),\label{2coeff}\\((1+A^4)(k-q_A)-4A^2q_A)a_k&=&A^2(k+2+q_A)a_{k+2}+A^2(k-2+q_A)a_{k-2},\qquad k\ge3.\nonumber\end{eqnarray}Exactly as described in \cite{OS}, the fourth relation can be written as $A_k=n_kA_{k-2}$ for $k\ge3$, where $A_k=a_{k+2}-A^2a_k$ and $$n_k=\frac{k-2+q_A}{A^2(k+2+q_A)}.$$This relies on the special, known form of $q_A$, which gives that $$1-A^4=(1+A^2)(1-A^2)=(1+A^2)^2q_A.$$Indeed, \begin{eqnarray*}A_k&=&n_kA_{k-2}\\\iff a_{k+2}-A^2a_k&=&n_k(a_k-A^2a_{k-2})\\\iff A^2(k+2+q_A)a_{k+2}-A^4(k+2+q_A)a_k&=&(k-2+q_A)a_k-A^2(k-2+q_A)a_{k-2}\\\iff A^2(k+2+q_A)a_{k+2}+A^2(k-2+q_A)a_{k-2}&=&\big((1+A^4)k+(1+A^4)q_A-2(1-A^4)\big)a_k\\&=&\big((1+A^4)k+(1+A^4)q_A-2(1+A^2)^2q_A\big)a_k\\&=&\big((1+A^4)(k+q_A)-2(1+A^4)q_A-4A^2q_A\big)a_k,\end{eqnarray*}which is exactly the fourth equation above. Since $n_k\to 1/A^2>1$ as $k\to\infty$, and we require boundedness on $a_k$ and thus $A_k$, it follows that necessarily $A_k=0$ for $k\ge1$. Hence $a_{k+2}=A^2a_k$ for $k\ge1$.

Now, using \eqref{0coeff} to write $C_A(\theta)$ in terms of $a_2$ and using $a_4=A^2a_2$, \eqref{2coeff} is equivalent to $$\left(2-2A^4-q_A(1+4A^2+2A^4)+\frac{2A^4(2+q_A)}{1+4A^2+A^4}\right)a_2=0.$$Using the given form for $q_A$, a tedious calculation shows that this is in turn equivalent to $$\frac{(1+A^2)^3}{1+4A^2+A^4}a_2=0,$$which itself is equivalent to $a_2=0$. It then follows that $C_A(\theta)=0$, and so \eqref{1coeff} is equivalent to $$\left((1+A^4)(1-q_A)-4A^2q_A\right)a_1-A^2\left((1+q_A)+A^2(3+q_A)\right)a_1=0.$$A slightly less lengthy, but similarly tedious, calculation shows that this is equivalent to $$-4A^2a_1=0,$$which, for $A\ne0$ is equivalent to $a_1=0$. Thus, we obtain that when $0\ne A\in(-1,1)$, the only solutions are given by $a_1=a_2=0$, whence it follows that $a_k=0$ for all $k$. Therefore, we conclude that $\mathrm{d}G[\theta_A]\theta=0$ if and only if $\theta=0$, which is to say that $\mathrm{d}G[\theta_A]$ is injective.\end{proof} We are now in a position to be able to present the proof of the theorems, which, with the use of this lemma will follow almost immediately. In essence, the proof of the theorems is an application of the Implicit Function Theorem, as stated above. Notice that it follows from Lemmas \ref{FinjGinj} and \ref{Ginj} that $\p_w\mathfrak{F}[0,\be_A,w_A]$ is an injective linear operator for each $0\ne A\in(-1,1)$. Therefore we shall prove the theorem only for $\F$, observing that the proof for $\mathfrak{F}$ is entirely similar.\begin{proof}[Proof of Theorems] As should be clear, the main result will follow from an application of the Implicit Function Theorem. It simply remains to verify the conditions of that theorem. First, we observe that $\p_w\F[\al,\be,w]$ can be written as a compact perturbation of an invertible linear operator, since the derivative of a compact nonlinear mapping is a compact linear operator. Thus, by \cite[Theorem 2.7.6]{BufTol}, it is a Fredholm operator of index zero. For any $0\ne A\in(-1,1)$ we have, from Lemmas \ref{FinjGinj} and \ref{Ginj}, that $\p_w\F[0,\be_A,w_A]$ is injective; by the Fredholm property it is also surjective and thus an isomorphism. Then, a simple application of the implicit function theorem yields, for each $0\ne A\in(-1,1)$, a neighbourhood $V_A\subset\R\times\R_+\times X$ of $(0,\be_A,w_A)$, and a (real-analytic) function $W_A:U_A\to X$ for some neighbourhood $U_A\subset\R\times\R_+$ of $(0,\be_A)$ such that, exactly as written in the statement of the theorem, $$\{(\al,\be,w)\in V_A:\F(\al,\be,w)=0\}=\{(\al,\be,W(\al,\be):(\al,\be)\in U\}$$and so that $W(0,\be_A)=w_A$. For the second part of the theorem, observe that those solutions in $V_A$ with $\al=0$ are given by $(0,\be,W(0,\be))$ for $(0,\be)\in U_A$. That is, they form a one-dimensional curve which, as explained above, must be exactly Crapper's waves and no others.\end{proof}The main result of this theorem is that in a neighbourhood of $(0,\be_A,w_A)$ in $\mathbb{R}\times\mathbb{R}_+\times X$ the solution set of $\F=0$ is parameterised by $(\al,\be)$ and is thus a two-dimensional manifold, or `sheet', of solutions, into which is embedded the curve of Crapper's waves; this curve corresponding exactly to those solutions with $\al=0$. We thus obtain a two-dimensional sheet of solutions with $\al>0$ and thus physically-relevant solutions to $\F=0$ which are close to Crapper's waves, but fit within the full gravity capillary regime.

\end{document}